\documentclass[12pt]{amsart}
\usepackage{inputenc}
\usepackage{amsmath,amssymb,amsthm}
\usepackage{graphicx,epsfig}
\usepackage{enumerate}
\usepackage{tikz}
\usepackage{multicol}
\usetikzlibrary{calc}
\usepackage{url}
\usepackage{hyperref}
\numberwithin{equation}{section}
\theoremstyle{plain}
\newtheorem{thm}{Theorem}
\newtheorem*{thm-main}{Main Theorem}
\newtheorem{lem}{Lemma}
\newtheorem{dfn}{Definition}
\newtheorem{cor}{Corollary}

\newtheorem{prop}{Proposition}
\newtheorem{conj}{Conjecture}
\newtheorem*{conjecture}{Conjecture}

\newtheorem{rem}{Remark}

\newtheorem{notation}{Notation}
\newcommand{\reg}[1]{\regularity{#1}}
\DeclareMathOperator{\regularity}{reg}
\DeclareMathOperator{\proj}{proj}
\DeclareMathOperator{\adj}{adj}
\DeclareMathOperator{\spec}{spec}
\DeclareMathOperator{\Mat}{Mat}
\DeclareMathOperator{\Ima}{Im}
\DeclareMathOperator{\p}{\mathfrak{p}}
\DeclareMathOperator{\Stab}{Stab}
\DeclareMathOperator{\en}{end}

\DeclareMathOperator{\mm}{\mathfrak{m}}

\begin{document}
\title{Equations of some embeddings of a projective space into another one}
\author{Marc Chardin}
\address{Institut de math\'{e}matiques de Jussieu, CNRS \& Sorbonne Universit\'e, 4 place Jussieu, 75005 Paris , France}
\email{marc.chardin@imj-prg.fr}

\author{Navid Nemati}
\address{ Universit\'{e} C\^{o}te d'Azur, Inria, 2004 route des Lucioles, 06902 Sophia Antipolis, France}
\email{navid.nemati@inria.fr}
\begin{abstract}
In \cite{EHU}, Eisenbud, Huneke and Ulrich conjectured a result on the Castelnuovo-Mumford regularity of the embedding of a projective space $\mathbb{P}^{n-1}\hookrightarrow \mathbb{P}^{r-1}$ determined by generators of a linearly presented $\mm$-primary ideal. This result implies in particular that the image is scheme defined by equations of degree at most $n$. In this text we prove that the ideal of maximal minors of the Jacobian dual matrix associated to the input ideal defines the image as a scheme; it is generated in degree $n$. Showing that this ideal has a linear resolution would imply that the conjecture in \cite{EHU} holds. Furthermore, if this ideal of minors coincides with the one of the image in degree $n$ - what we hope to be true - the linearity of the resolution of this ideal of maximal minors is equivalent to the conjecture in \cite{EHU}.
\end{abstract}
\maketitle
\section{Introduction and Preliminaries}
Let $\phi: \mathbb{P}^{n-1}\dashrightarrow \mathbb{P}^{r-1}$ be a rational map defined by homogeneous  forms $f_1,\dots, f_r$ and $I$ be the ideal generated by these forms. The Rees algebra associated to the ideal $I$, $\mathcal{R}(I)$, gives the bihomogeneous  coordinate ring of the closure of the graph of $\phi$ inside $\mathbb{P}^{n-1}\times \mathbb{P}^{r-1}$. The special fiber ring $\mathcal{F}(I)$ is the coordinate ring of the closed image of $\phi$. The rings $\mathcal{R}(I)$ and $\mathcal{F}(I)$ are \textit{blowup rings} associated to $I$. It is a fundamental problem to find the defining equations of the Rees ring from generators of $I$, for significant classes of ideals. This problem has been studied extensively by commutative algebraists, algebraic geometers, and, more recently, by applied mathematicians in geometric modeling (see \cite{BCS, Cox, KUP2, KUP1, KUP3}). Rees algebras provide an algebraic realization for the concept of blowing up a variety along a subvariety.

Let $S= k[x_1,\dots,x_n]$ be a polynomial ring over a field $k$ and $I$ be an ideal of $S$. The Rees algebra $\mathcal{R}(I)$ of an ideal $I$ is defined as 
$$
\mathcal{R}(I)= S[It]= S\oplus It\oplus I^2t^2\oplus \cdots .
$$
The equations of the Rees algebra are elements in  the kernel of the epimorphism \linebreak$\Phi : S[T_1, \dots , T_r]\rightarrow \mathcal{R}(I)$ given by $\Phi(T_i)= f_it$ where $I= (f_1,\dots, f_r)$. These equations do depend on the choice of generators of $I$; however, their number and degrees do not, if one chooses minimal generators of a graded ideal $I$. 

An important object for the study of Rees algebras is the symmetric algebra, $\text{Sym}(I)$, of an ideal $I$. It could be defined as
$$
\text{Sym}(I)= S[T_1,\dots,T_r]/\mathcal{L},
$$
where $\mathcal{L}= ([T_1,\dots,T_r]\cdot M)$ and $M$ is a presentation matrix of $I$. The map $\Phi$ above
factors through the symmetric algebra. To determine equations of $\mathcal{R}(I)$  of degrees at least two in the $T_i$'s is equivalent to study the kernel of the
map $\text{Sym}(I)\rightarrow  \mathcal{R}(I)$. Traditionally, techniques for computing the defining ideal of
$\mathcal{R}(I)$ often revolved around the notion of \textit{Jacobian dual}. 

For a polynomial ring $S$ 
and an ideal $I$ with a presentation $S^N\xrightarrow{M} S^r\rightarrow I \rightarrow 0$, the \textit{Jacobian dual} of $M$ is defined to be a matrix $\Theta(M)$ (we use $\Theta$ if the context makes no confusion possible) with linear entries in $S[T_1, \dots , T_r]$ such that
$$
[T_1,\dots , T_r]\cdot M = [x_1,\dots , x_n]. \Theta(M).
$$
In this article, we focus on the Jacobian dual matrix of linearly presented $\mm$-primary ideals. The main theorem of this paper is the following:
\begin{thm-main}
Let $S=k[x_1,\dots ,x_n]$ be a polynomial ring over a field $k$, $I=(f_1,\dots,f_r)$ be a linearly presented $\mm$-primary ideal of $S$ and $\Theta$ be the Jacobian dual matrix of a presentation matrix of $I$. 
Then the image of the map $\mathbb{P}^{n-1}\rightarrow \mathbb{P}^{r-1}$ defined by the forms $f_1,\dots f_r$ is scheme-theoretically defined by the ideal of maximal minors of $\Theta$.
\end{thm-main}
It is worth mentioning that it is not hard to show that the image is defined as a set by the ideal of maximal minors (see Lemma \ref{radical}); the more delicate point is to show that the image is scheme-theoretically defined by the ideal of maximal minors.
As a consequence, the image of $\Phi$ is scheme-theoretically defined by equation of degree $n$. We finish the section by presenting two conjectures about the ideal of maximal minors of $\Theta$.

In the last section, we present the connection between our work and  a conjecture by Eisenbud, Huneke, and Ulrich in \cite{EHU}.
\begin{conjecture}\cite[Conjecture 1.1]{EHU}
Let $S= k[x_1,\dots ,x_n]$ and $I$ be an ideal of $S$. If $I$ is a linearly presented $\mm$-primary ideal generated in degree $d$, then $I^{n-1}= \mm^{d(n-1)}$. Equivalently, 
$$
\reg(I^{n-1}) = (n-1)d.
$$
\end{conjecture}
In particular, we show that this conjecture holds if the ideal of maximal minors of $\Theta$ has a linear resolution (see Theorem \ref{Thm- conjectures} ).

\section{Jacobian dual matrix}
For the rest of the article we adopt the following notations.
\begin{notation}\label{Notation}
Let $S=k[x_1,\dots, x_n]$ be a polynomial ring over a field $k$ and $I=(f_1,\dots,f_r)$ be a linearly presented $\mm$-primary ideal, $\phi : \mathbb{P}^{n-1}\dashrightarrow \mathbb{P}^{r-1}$ be the map
defined by $\phi (\textbf{x})= (f_1(\textbf{x}), \cdots , f_r(\textbf{x}))$ and denote 
the graph of $\phi$ by $\Gamma\subseteq \mathbb{P}^{n-1}\times \mathbb{P}
^{r-1}$. Set $\pi_1: \Gamma \rightarrow \mathbb{P}^{n-1} $ and $\pi_2: 
\Gamma \rightarrow \mathbb{P}^{r-1} $ for the two natural projections and $W:= \Ima(\pi_2)$. We write $B=k[T_1,\dots,T_r]$. 
The defining ideal $J$ of the symmetric algebra $\text{Sym}(I)$ is generated by the bilinear forms $L_i(\textbf{x},\textbf{T})= \sum 
a_i(\mathbf{x})T_i$ where $\sum a_i(\mathbf{x})f_i=0$ is 
a linear syzygy of $I$.  The Jacobian dual $\Theta$ is given by 
$$
\Theta
\begin{bmatrix}
x_1\\
\vdots\\
x_n\\
\end{bmatrix}
=
\begin{bmatrix}
L_1\\
\vdots\\
L_N\\
\end{bmatrix},
$$
where the $L_i$'s are minimal generators of $J$ (equivalently, the corresponding syzygies are minimal generators of the module of syzygies). Thus $\Theta$ is a matrix whose entries are linear forms in the $T_i$'s and we set $W':=\proj (\bar{B})$ where $\bar{B}= B/I_n(\Theta)$ and $I_n(\Theta)$ is the ideal of maximal minors of $\Theta$ (the $n\times n$ minors). Finally, if $p\in W\subseteq W'$ and $\mathfrak{p}$ is the corresponding homogeneous 
prime in $\proj(B)$, we denote by $M_p$ the image of a matrix $M$ by the natural map obtained by applying the map $B\rightarrow B_{\mathfrak{p}}/I_r (M)_{\mathfrak{p}}$ to its entries, where $r$ is the rank of $M$.
\end{notation}
\begin{dfn}
Let $M$ be a $p\times q$ matrix over $B$. For a prime ideal $\p\supseteq I_q (M)$, consider the natural map $B\rightarrow (B/I_q(M))_{\p}$ and its natural extension $\Mat_{p\times q}(B)\rightarrow \Mat_{p\times q}((B/I_q(M))_{\p})$. We write $M_{\p}$ for the image of $M$ under this last map.
\end{dfn}
\begin{lem}\label{radical}
With the above notations,  $\sqrt{I_n(\Theta)}= I(W)$.
\end{lem}
\begin{proof}
As $I$ is $\mathfrak{m}$-primary, $\Gamma\subseteq \mathbb{P}^{n-1}\times \mathbb{P}^{r-1}$ is scheme defined by the symmetric algebra, and is in particular, set defined by $J$. Consider the projection $\pi_2: \mathbb{P}^{n-1}\times \mathbb{P}^{r-1} \rightarrow \mathbb{P}^{r-1}$ and $p= (t_1:\dots : t_{r})\in \mathbb{P}^{r-1}$. Let $\Theta(p)$ be the image of $\Theta$ by the specialization $T_i\mapsto t_i$. Then,
\begin{align*}
p\in W & \Leftrightarrow \exists x= (x_1,\dots ,x_n) ; (x,p)\in \Gamma\\
&\Leftrightarrow \exists x; L_i(x,t)=0 \, \forall i\\
& \Leftrightarrow \exists x ; \Theta (p)
\begin{bmatrix}
x_1\\
\vdots\\
x_n\\
\end{bmatrix}=0
\\
&\Leftrightarrow t\in I_n(\Theta).
\end{align*}
\end{proof}
\begin{lem}\label{fiber}
Assume that $\psi :\mathbb{P}^{n-1}\setminus X \rightarrow \mathbb{P}^{r-1}$ is
defined by $I=(f_1,\dots ,f_r)$ where $X=V(I)$. Let $\Gamma\subseteq \mathbb{P}^{n-1}\times \mathbb{P}^{r-1}$ be the closure of the graph of $\psi$. If $\pi : \Gamma \rightarrow \overline{\Ima (\psi)}$ is the projection map, then $ \forall p\in \overline{\Ima (\psi)}$, setting $\dim \emptyset :=-1$,
$$
\dim \pi^{-1}(p)\leq \dim X+1.
$$
\end{lem}
\begin{proof}
Let $p = (p_1:\dots : p_r)\in \overline{\Ima \psi}$ and let $V$ be an irreducible component of $\pi^{-1}(p)$. If $V$ is contained in $X$ it has dimension at most the one of $X$. Else $\psi |_{V}: V\setminus X\rightarrow \mathbb{P}^{r-1}$ is a constant map with image $p$. Hence for all $1\leq i,j\leq r$, $(f_ip_j-f_jp_i)\in I(V)$. Choose $i$ with $p_i\neq 0$, then $f_j- p_jf_i/p_i\in I(V)$ for all $j$, and therefore $I\subseteq I(V)+(f_i)$, and the height of $I(V)+(f_i)$ is the one of $I(V)$ plus one.
\end{proof}
\begin{lem}\label{rank}
Adopt Notation \ref{Notation}. For a point $p\in W\subseteq W'$, the rank of $\Theta_{\p}$ is $n-1$.
\end{lem}
\begin{proof}
By Lemma \ref{fiber}, $\pi^{-1}(p)$  is a zero dimensional scheme. The assertion follows from the fact that this fiber is a linear space defined by the system of equations 
$$
\Theta_{\p}(p)\begin{bmatrix}
x_1\\
\vdots\\
x_n\\
\end{bmatrix}=0.
$$ 
\end{proof}
\begin{rem}
Since $\reg I^t= dt+0$ for $t\gg 0$, the regularity of fibers of the projection \linebreak$\pi_2: \Gamma \rightarrow \mathbb{P}^{r-1}$  at any point is zero \cite{EH}. The only finite set of point(s) with zero regularity is a  single reduced point. More generally, the regularity of stalks of $\pi_2$ at any point is zero too \cite{CH}.
\end{rem}
\begin{dfn} \label{delta}
Let $N$ be a $m\times n$ matrix. For $1\leq r\leq n$ and $1\leq s\leq m$, define  $N^{(i_1,\dots i_r)}_{(j_1,\dots, j_s)}$ be the matrix obtained by deleting $i_1,\dots,i_r$-th columns and $j_1,\dots j_s$-th rows of $N$. Let $N$ be a $(n-1)\times n$ matrix, denote $\Delta_i(N)$ to be the $(-1)^i \det N^{(i)}$, with $ N^{(i)}$ the  matrix obtained by deleting $i$-th column of $N$. For simplicity we write $\Delta_i$ instead of $\Delta_i(N)$ if it is clear what matrix we consider.
%Let $N$ be a $(n-1)\times n$ matrix. Define $N^{(i)}_{(j)}$ be the matrix obtained by deleting $i$-th column and $j$-th row of $N$ and $\Delta_i$ to be the $(-1)^i \det N^{(i)}$, with $ N^{(i)}$ the  matrix obtained by deleting $i$-th column of $N$. 
\end{dfn}
\begin{lem}\label{adj}
Let $N$ be a $(n-1)\times n$ matrix, then

$$
(-1)^{i}\adj(N^{(i)})N=\left[\begin{array}{ccc|c|ccc}
\Delta_{i}&&&-\Delta_1&&&\\
&\ddots&&\vdots&&&\\
&&\Delta_{i}&-\Delta_{i-1}&&&\\ \hline
&&&-\Delta_{i+1}&\Delta_{i}&&\\
&&&\vdots&&\ddots&\\
&&&-\Delta_n&&&\Delta_{i}\\
\end{array}	\right] ,
$$
where only the non zero terms are displayed.
	\end{lem}
\begin{proof}
The $(j,\ell)$-entry of $\adj(N^{(i)})N$ is 
$
\sum_k b_{j,k}a_{k,\ell}
$ where $b_{j,k}$ and $a_{k,\ell}$ are the entries of $\adj(N^{(i)})$ and $N$. 

Notice that $b_{j,k}=(-1)^{j+k}\det N^{(i,j)}_{(k)}$ if $j<i$ and $b_{j,k}=(-1)^{j+k}\det N^{(i,j+1)}_{(k)}$ else.

Hence then the $(j,\ell)$-entry of $\adj(N^{(i)})N$ is
$$
\sum_k (-1)^{j+k}\det N^{(i,j)}_{(k)}a_{k,\ell}\ {\rm if}\ j<i\quad {\rm and}\ \sum_k (-1)^{j+k}\det N^{(i,j+1)}_{(k)}a_{k,\ell}\ {\rm if}\ j\geq i .
$$

First let $\ell <i$. If $\ell=j$, then it is equal to $\det N^{(i)}= (-1)^i \Delta_i$. If $\ell \neq j<i$, replace the $j$-th column of $N$ with its $\ell$-th column and call it $N'$. By expanding $\Delta_i(N')$ along the $j$-th column we get
$$
0= \Delta_i(N') = (-1)^i \sum_k (-1)^{j+k}\det N^{(i,j)}_{(k)}a_{k,\ell},
$$
and similarly if $j\geq i$, by relacing the $(j+1)$-th column of $N$ with its $\ell$-th column.

Second, if $\ell>i$, similar arguments as in the case $\ell <i$ show that if $j=\ell -1$, the $(j,j+1)$-entry is equal to $\Delta_i$, and the $(j,\ell)$-entry is equal to $0$ if $j\neq \ell -1$.

Thirdly, if $j<\ell=i$,  the $(j,i)$-entry of $\adj(N^{(i)})N$ is  
$$
\sum_k (-1)^{j+k}\det N^{(i,j)}_{(k)}a_{k,i}.
$$
By expanding $\Delta_j = (-1)^{j}\det N^{(j)}$ along the $i$-th column we get 
$$
\Delta_j = (-1)^{j} \sum_k (-1)^{i-1+k} a_{k,i} \det N^{(i,j)}_{(k)}= (-1)^{i-1} \sum_k (-1)^{j+k} a_{k,i} \det N^{(i,j)}_{(k)}.
$$
Finally, if $ j\geq \ell=i$ then $b_{j,k}= (-1)^{j+k}\det N^{(i,j+1)}_{(k)}$ so $(j,i)$-th entry $\adj(N^{(i)})N$ is  
$$
\sum_k (-1)^{j+k}\det N^{(i,j+1)}_{(k)}a_{k,i}.
$$ 
By expanding $\Delta_{j+1} = (-1)^{j+1}\det N^{(j+1)}$ on the $i$-th column we get
$$
\Delta_{j+1} = (-1)^{j+1} \sum_k (-1)^{i+k} a_{k,i} \det N^{(i,j+1)}_{(k)}= (-1)^{i-1} \sum_k (-1)^{j+k} a_{k,i} \det N^{(i,j+1)}_{(k)}.
$$
\end{proof}
\begin{lem}\label{3lemmas}
For a $1\times n$ row vector $R$ define $L$ as a product of $R$ and column vector $(x_1,\dots, x_n).$
Let $N$ be a $(n-1)\times n$ matrix with row vectors $R_1,\dots R_{n-1}$ and consider  corresponding linear forms $L_1,\dots,L_{n-1}$. Let $M$ obtained by bordering $N$ with a row vector $R_n$ and corresponding linear form $L_n$. Then for $1\leq i,j\leq n$
%Let $N$ be a $(n-1) \times n$ matrix. Denote $L_i$ by the $i$-th entries column vector. 

\begin{itemize}

\item[(1)] $\Delta_i L_n + (-1)^{n+1} \det(M) x_i \in (L_1,\dots , L_{n-1})$
\item[(2)] $(x_i\Delta_j - x_j\Delta_i) \in (L_1,\dots , L_{n-1})$.
\item[(3)] If $j\neq n$, $\Delta_i L_j \in (x_1\Delta_i -x_i\Delta_1 ,\ldots , x_n \Delta_i -x_i\Delta_n) $.
\end{itemize}
\end{lem}
\begin{proof}
(1) Consider 
$$
\det (M) \begin{bmatrix}
x_1\\
\vdots\\
x_n
\end{bmatrix}=
\adj(M)M\begin{bmatrix}
x_1\\
\vdots\\
x_n
\end{bmatrix}= \adj(M) \begin{bmatrix}
L_1\\
\vdots\\
L_n
\end{bmatrix}.
$$
The $(i,1)$-entries in the above matrix are $\det(M)x_i $.
On the other, it is equal to 
$$
\sum_{1\leq k\leq n} (-1)^{k+i} \det M^{(i)}_{(k)} L_k= \sum_{1\leq k\leq n-1} (-1)^{k+i} \det M^{(i)}_{(k)} L_k+ (-1)^{n}\Delta_i L_n.
$$
(2) By choosing $L_n= e_j$ in (1) we get
\begin{align*}
\Delta_ix_j + (-1)^{n+1}\det(M)x_i &= \Delta_ix_j + (-1)^{n+1} (-1)^{n+j}(-1)^j\Delta_jx_i \\
&=\Delta_ix_j - x_i\Delta_j\in (L_1,\dots, L_{n-1}).
\end{align*}
(3) 
By Lemma \ref{adj} 
$$
\Delta_i \begin{bmatrix}
L_1\\
\vdots \\
L_{n-1}
\end{bmatrix}=
(-1)^iN^{(i)}\adj(N^{(i)})N\begin{bmatrix}
x_1\\
\vdots \\
x_n
\end{bmatrix}
=N^{(i)}
\begin{bmatrix}
x_1\Delta_i-x_i \Delta_1\\
\vdots \\
x_{i-1}\Delta_i-x_i \Delta_{i-1}\\
x_{i+1}\Delta_i-x_i \Delta_{i+1}\\
\vdots \\
x_{n}\Delta_i-x_i \Delta_{n}
\end{bmatrix} .
$$
\end{proof}

\begin{cor}\label{corollary}
For any $(n-1)\times n$ submatrix of $\Theta$ which corresponds to $L_{i_1},\dots, L_{i_{n-1}}$ define $\Delta_i$ for $1\leq i\leq n$ as in Definition \ref{delta}. If $\Delta_i\neq 0 $ for some $i$, then
\begin{itemize}
\item[(1)] $\Delta_i J\subseteq (L_{i_1}, \dots , L_{i_{n-1}})$ modulo $I_n(\Theta)$,
\item[(2)] $(L_{i_1},\dots, L_{i_{n-1}})B_{(\Delta_i)}= I_2 \begin{bmatrix}
x_1 & \cdots & x_n\\
\Delta_1 & \cdots & \Delta_n
\end{bmatrix} B_{(\Delta_i)}$.
\end{itemize}
\end{cor}
\begin{proof}

Choose for $N$ the submatrix of $\Theta$ such that 
$$
N\begin{bmatrix}
x_1\\
\vdots\\
x_n
\end{bmatrix}= \begin{bmatrix}
L_{i_1}\\
\vdots\\
L_{i_{n-1}}
\end{bmatrix} .
$$

(1)  follows from Lemma \ref{3lemmas} (1), by choosing 
for $M$ the matrix corresponding to add any minimal generator $L_j$ of $J$.

(2) follows from Lemma \ref{3lemmas} (2) and (3).
\end{proof}
\begin{thm}\label{main-sat}
Let $S=k[x_1,\dots ,x_n]$ be a polynomial ring over a field $k$, $I=(f_1,\dots,f_r)$ be a linearly presented $\mm$-primary ideal of $S$ and $\Phi : \mathbb{P}^{n-1}\dashrightarrow \mathbb{P}^{r-1}$ be the map defined by the forms $f_1,\dots f_r$. Let $\Theta$ be the Jacobian dual matrix of a presentation matrix of $I$,  then 
$$
I_n(\Theta)^{sat}= I(W).
$$
\end{thm}
\begin{proof}
By Lemma \ref{radical}, $\sqrt{I_n(\Theta)}= I(W)$. Let $\p \in \spec(B)$ containing $I_n(\Theta)$. We need to show that $(B/I_n(M))_{\p} \cong (B/I(W))_{\p}$. For this to hold, it suffices to prove that $\bar{B}_{\p}=(B/I_n(M))_{\p}$ is a domain. By Lemma \ref{rank}, there exist $L_{i_1},\dots L_{i_{n-1}}$ and $i$ with $\Delta_i \notin \p$, with notations as in Corollary \ref{corollary}. Now
\begin{align*}
\bar{B}_{\p}[x_1,\dots, x_n]/J\otimes_{B}\bar{B}_{\p}&= \bar{B}_{\p}[x_1,\dots, x_n]/(L_{i_1},\dots L_{i_{n-1}}) \otimes_{B}\bar{B}_{\p}\qquad \mbox{by Corollary \ref{corollary} (1)}\\
&= \bar{B}_{\p}[x_i] \qquad  \qquad \qquad \qquad \qquad \quad \qquad \qquad  \mbox{by Corollary \ref{corollary} (2).}
\end{align*}

As $J$ scheme defines $\Gamma$, $\proj ( \bar{B}_{\p}[x_1,\dots, x_n]/J\otimes_{B}\bar{B}_{\p})$ is the stalk of  the isomorphism \linebreak$\pi :\Gamma \rightarrow W$ over $V(\p )$; in particular, it is reduced and irreducible. This shows that $\proj(\bar{B}_p[x_i])\cong \spec(\bar{B}_{\p})$ is reduced and irreducible. Hence $\bar{B}_{\p}$ is a domain.
\end{proof}

We finish this section with two conjectures.
\begin{conj}\label{conj-truncation}
Let $S=k[x_1,\dots ,x_n]$ be a polynomial ring and $I=(f_1,\dots,f_r)$ be a linearly presented $\mm$-primary ideal of $S$. Suppose that $\Phi : \mathbb{P}^{n-1}\dashrightarrow \mathbb{P}^{r-1}$ is a rational map defined by forms $f_1,\dots f_r$. Let $\Theta$ be the Jacobian dual matrix of a presentation matrix of $I$, then 
$$
I_n(\Theta)= I(W)_{\geq n}.
$$
\end{conj}
\begin{conj}\label{conj-linear resolution}
Let $S=k[x_1,\dots ,x_n]$ be a polynomial ring and $I=(f_1,\dots,f_r)$ be a linearly presented $\mm$-primary ideal of $S$.  Let $\Theta$ be the Jacobian dual matrix of a presentation matrix of $I$ then $I_n(\Theta)$ has a linear resolution.
\end{conj}

\section{Asymptotic behavior of regularity }
In this section, we study the conjecture of Eisenbud, Huneke, and Ulrich \cite[Conjecture 1.1]{EHU} on the asymptotic behavior of regularity of linearly presented $\mm$-primary ideals. We will show the relation between this conjecture and the ideal of maximal minors of Jacobian dual matrices. We start this section by stating two equivalent definition of Castelnuovo-Mumford regularity. Before that, we need to mention the definition of graded Betti numbers. For a finitely generated graded $S$-module $M$, a minimal graded free resolution of $M$ is an exact sequence
$$
0\rightarrow F_p \rightarrow F_{p-1} \rightarrow \cdots \rightarrow F_0 \rightarrow M \rightarrow0,
$$
where every $F_i$ is a graded free $S$-module of the form
$F_i =\oplus_{j\in \mathbb{Z}}S(-j)^{\beta_{i,j}(M)}$
with the minimal number of basis elements, and every map is graded. The value $\beta_{i,j(M)}$ is called the $i$th graded Betti number of $M$ of degree $j$. 
\begin{thm}
Let $M$ be a finitely generated $S$-module,  the regularity of $M$ is given by:
$$
\reg(M)= \max_i \lbrace \en(H^i_{\mm}(M))+i \rbrace.
$$
Although the regularity is a measure of vanishing of local cohomologies, it is also a measure of vanishing of graded Betti numbers. Eisenbud and Goto in \cite{EG} proved that
$$
\reg(M) = \max_i \lbrace  j-i | \beta_{i,j}(M)\neq 0\rbrace.
$$
\end{thm}
The most significant simple result on the regularity of powers of graded ideals is the following one, due independently to  Kodiyalam \cite{Kodiyalam} and to Cutkosky, Herzog and Trung \cite{CHT}. 
\begin{thm}\label{linearity powers}
Let $I$ be an ideal of $S =k[x_1,\dots,x_n]$.  There exists $t_0$ and $b$ such that 
$$
\reg(I^t)= td+b, \quad \forall t\geq t_0
$$
with 
$$
d:= \min \lbrace \mu \,\, \vert \,\, \exists t\geq 1, \,\, (I_{\leq \mu}) I^{t-1}=I^t\rbrace.
$$ 
\end{thm}
Notice that when $I$ is equigenerated the number $d$ in the above theorem is the degree of the generators of $I$.  We call smallest such $t_0$ as the stabilization index of $I$ and denote it by $\Stab (I)$.  In \cite{EHU}, authors showed that if $I$ is a linearly presented $\mm$-primary ideal, then the powers of $I$ eventually have a linear resolution.
\begin{thm}[\cite{EHU}]
Let $S= k[x_1,\dots ,x_n]$ and $I$ be an ideal of $S$. If $I$ is  a linearly presented $\mm$-primary ideal generated in degree $d$, then $\reg(I^t) = td$ for $t\gg 0$. 
\end{thm}
In addition, they conjectured an upper bound for the stabilization index of $I$.
\begin{conj}\label{main conjecture}\cite[Conjecture 1.1]{EHU}
Let $S= k[x_1,\dots ,x_n]$ and $I$ be an ideal of $S$. If $I$ is a linearly presented $\mm$-primary ideal generated in degree $d$, then $I^{n-1}= \mm^{d(n-1)}$. In other word, 
$$
\reg(I^{n-1}) = (n-1)d.
$$
\end{conj}
\begin{rem}\label{proved cases}
Eisenbud, Huneke and Ulrich  in \cite{EHU} proved this conjecture when $n=3$ or $I$ is a monomial ideal.
\end{rem}
\begin{rem}\label{rem- Hilbert function}
Since the only $\mm$-primary ideals with linear resolution are the powers of $\mm$, for an $\mm$-primary ideal $I$ generated in degree $d$,
$$
\Stab(I) = \min \lbrace t \, \vert  h_{I^t}(dt)= h_{\mm^t}(dt)\rbrace,
$$
where $h_M(d):= \dim _{k} M_d$ is the Hilbert funtion of $M$ at degree $d$. 
\end{rem}
The following proposition provide a connection between   regularity of powers of an ideal $I=(f_1,\dots,f_r)$ and the image of $\phi : \mathbb{P}^{n-1}\dashrightarrow \mathbb{P}^{r-1}$ defined by $f_1,\dots,f_r$.
\begin{prop}\label{upper bound regularity}
Let $S= k[x_1,\dots ,x_n]$ and $I$ be an ideal of $S$. Adopting the Notation \ref{Notation}, If $I$ is a linearly presented $\mm$-primary ideal generated in degree $d$, then
\begin{itemize}
\item[(1)] $W$ is smooth of dimension $n-1$, and $e(W) = d^{r-1}$.
\item[(2)]
%$
%\reg{I(W)} \leq \min \big \lbrace a | H_W(a)= \binom{da+n-1}{n-1}\big %\rbrace+1
%$.
$
\reg I(W) = \max\lbrace \Stab(I)+1, \reg I(V_n^{(d)}) \rbrace = \max \{ \Stab(I)+1, n- \lceil \dfrac{n}{d}\rceil \}
$
\end{itemize}
where $V_n^{(d)}$ is the Veronese embedding of degree $d$. 
\end{prop}
\begin{proof}
Part $(1)$ follows from \cite[Proposition 1.7(b)]{COR}. For proving part $(2)$ we use the definition of regularity via local cohomologies. 
Let $I(W)\subset B= k[\mathbf{T}]$ and $I(V_n^{(d)})
\subset B'=k[\mathbf{T}']$ with $B\subset B'$ and $I(W)=I(V_n^{(d)})\cap B$. 
As the natural map $ B/I(W)\rightarrow B'/I(V_n^{(d)})$ is an isomorphism locally on the punctured spectrum, $H^i_{\mathbf{T}}(B/I(W))\cong H^i_{\mathbf{T}}(B'/I(V_n^{(d)}))\cong H^i_{\mathbf{T}'}(B'/I(V_n^{(d)}))$ for $i\geq 2$ and there is an exact sequence:
$$
0\rightarrow B/I(W)\rightarrow B'/I(V)\rightarrow H^1_{\mathbf{T}}(B/I(W))\rightarrow 0.
$$
Since $h_{B/I(W)}(a)= h_{I^a}(da)$ and $h_{B/I(V)}(a)= h_{\mm^{da}}(da)$, by Remark \ref{rem- Hilbert function} and above exact sequence,  $H^1_{\mathbf{T}}(B/I(W))=0$  if and only if $i\geq \Stab (I)$. In other words,  $\text{end} H^1_{\mathbf{T}}(B/I(W))= \Stab (I)-1$. The  assertion follows from the definition of regularity and the fact that $\reg I(W)= \reg B/I(W)+1$.
\end{proof}
Conjecture \ref{main conjecture} was the initial motivation of this work, we end this paper by showing that how our conjectures are related to Conjecture \ref{main conjecture}.
\begin{thm}\label{Thm- conjectures}
Let $S= k[x_1,\dots ,x_n]$ and $I$ be an ideal of $S$. Adopting the Notation \ref{Notation}, If $I$ is a linearly presented $\mm$-primary ideal generated in degree $d$, then
\begin{itemize}
\item[$(1)$] Conjecture \ref{main conjecture} follows from Conjecture \ref{conj-linear resolution}.
\item[$(2)$] If Conjecture \ref{conj-truncation} holds, then Conjecture \ref{conj-linear resolution} and Conjecture \ref{main conjecture} are equivalent.
\end{itemize}
\end{thm}
\begin{proof}
$(1)$ Assume $I_n(\Theta)$ has a linear resolution, i.e. $\reg(I_n(\Theta))=n$. By Theorem \ref{main-sat}, $I_n(\Theta)^{sat}= I(W)
$.  Hence 
$$
\reg I(W)\leq  \reg(I_n(\Theta)) \leq n.
$$  
As $\reg(I(V_n^{(d)})=n$,  Proposition \ref{upper bound regularity} implies that  $\Stab(I)\leq n-1$. 

$(2)$ First we should note that by Proposition \ref{upper bound regularity}, $\Stab(I)\leq n-1$ if and only if $\reg(I(W))\leq n$. As we assume $I_n(\Theta)= I(W)_{\geq n}$,  $\reg(I(W))\leq n$ if and only if $I_n(\Theta)$ has a linear resolution.
\end{proof}

\bibliographystyle{acm} 
\bibliography{bib}

\end{document}